\documentclass[
final
]{dmtcs-episciences}

\usepackage[utf8]{inputenc}
\usepackage{subfigure}
\usepackage{amsmath,amssymb,amsthm,mathtools}
\usepackage[round]{natbib}

\newtheorem{theorem}{Theorem}[section]
\newtheorem{proposition}[theorem]{Proposition}
\newtheorem{lemma}[theorem]{Lemma}
\newtheorem{corollary}[theorem]{Corollary}
\theoremstyle{definition}
\newtheorem{definition}[theorem]{Definition}
\theoremstyle{remark}
\newtheorem{remark}[theorem]{Remark}

\author[N. B. Huamaní]{N. B. Huamaní\affiliationmark{1}}
\title[Vertex degrees in grid graphs associated with 213-avoiding permutations]{Vertex degrees in grid graphs associated with 213-avoiding permutations}
\affiliation{
  Department of Mathematics and Physics, Universidad Nacional de San Crist\'obal de Huamanga, Per\'u}

\begin{document}
\publicationdata
{vol. 25:3 special issue for main purpose}
{2026}
{1}
{10.46298/dmtcs.10472}
{2026-1-26; None}
{2026-1-26}
\maketitle
\begin{abstract}
Given a permutation of size $n$, we consider its associated grid graph whose $i$th column has height equal to the $i$th entry, with vertical edges between consecutive levels and horizontal edges between equal levels in adjacent columns. We study global degree statistics of these graphs when the permutation is chosen from the Catalan avoidance class $\mathrm{Av}_n(213)$ (and, by reversal, also from $\mathrm{Av}_n(312)$).

We first obtain an explicit closed form for the total number of horizontal edges summed over all permutations in $\mathrm{Av}_n(213)$. We then determine, for each degree $r\in\{1,2,3,4\}$, the total number of degree-$r$ vertices accumulated over the same class, yielding closed expressions in terms of central binomial coefficients and powers of four. The proofs rely on the Catalan decomposition induced by the position of the minimum entry, which leads to gluing identities and algebraic functional equations for ordinary generating functions, completed using global vertex and degree-sum identities.

As a consequence, we derive asymptotic degree proportions for a uniform random permutation in $\mathrm{Av}_n(213)$: the distribution concentrates and the proportion of degree-$4$ vertices tends to $1$, with a deficit of order $n^{-1/2}$.
\end{abstract}

\section{Introduction}\label{sec:intro}

We study global degree statistics in the grid graphs associated with permutations, restricting the underlying permutation to a Catalan
avoidance class.

Let $n\in\mathbb{N}$ and $[n]=\{1,2,\dots,n\}$. Let $S_n$ denote the set of all permutations of $[n]$,
that is, all words $\pi=\pi_1\pi_2\cdots \pi_n$ containing each element of $[n]$ exactly once.
Following \citet{BK}, to each $\pi\in S_n$ we associate a \emph{grid graph} $G_\pi$ with $n$ columns:
column $i$ has height $\pi_i$; we place vertical edges between consecutive levels within each column and horizontal edges between
equal levels in adjacent columns. For $n\ge2$, every vertex of $G_\pi$ has degree in $\{1,2,3,4\}$ \citep{BK}.

Two global quantities are immediate from the definition: the total number of vertices satisfies
$|V(G_\pi)|=\sum_{i=1}^n \pi_i=\binom{n+1}{2}$ (independent of $\pi$), and the total number of vertical edges is
$|E_{\rm vert}(G_\pi)|=\sum_{i=1}^n(\pi_i-1)=\binom{n}{2}$.
In contrast, the horizontal edges depend on $\pi$ and are given by the statistic
\[
H(\pi)=\sum_{t=1}^{n-1}\min\{\pi_t,\pi_{t+1}\}.
\]
In the unrestricted case (averaging over $S_n$), \citet{BK} obtain closed formulas for the total number of
horizontal edges and for the global vertex counts by degree, together with limiting proportions. Moreover, in the final section of
\citet{BK} they propose extending this program to avoidance classes $\mathrm{Av}_n(\tau)$, where $\tau$ is a pattern of length $3$.

In this paper we treat completely the Catalan class $\mathrm{Av}_n(213)$. Unlike the unrestricted case, pattern avoidance imposes a rigid
block structure: if the entry $1$ occurs in position $k$, then every value to the left of $1$ is larger than every value to the right,
forcing a Catalan gluing by lengths. This rigidity is reflected in the column profile of $G_\pi$.
Figure~\ref{fig:ex213} compares two examples in $\mathrm{Av}_4(213)$ with a permutation containing the pattern $213$.

\begin{figure}[t]
\centering
\IfFileExists{Figure213.pdf}{%
  \includegraphics[width=0.99\linewidth]{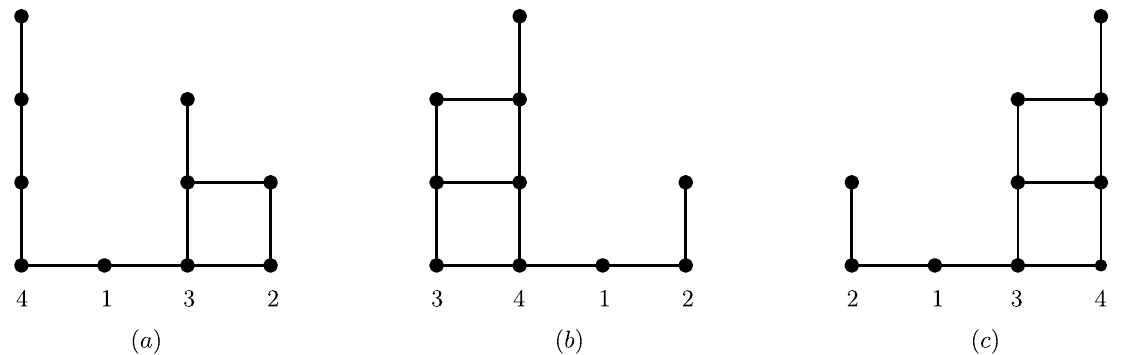}%
}{%
  \fbox{\parbox{0.92\linewidth}{\centering Figure file \texttt{Figure213.pdf} not found.\par
  (This placeholder allows the source to compile without the figure.)}}%
}
\caption{Grid graphs $G_\pi$ for $n=4$:
(a) $\pi=4132$, (b) $\pi=3412$ (both in $\mathrm{Av}_4(213)$), and
(c) $\pi=2134$, which contains the pattern $213$.}
\label{fig:ex213}
\end{figure}

Our main results are closed formulas for global totals as $\pi$ ranges over $\mathrm{Av}_n(213)$.
First, for the total number of horizontal edges
\[
H_n=\sum_{\pi\in\mathrm{Av}_n(213)}H(\pi)
\]
we obtain an explicit expression in terms of central binomial coefficients and powers of $4$.
Second, for each $r\in\{1,2,3,4\}$ we derive a closed form for
\[
Q_r(n):=\sum_{\pi\in\mathrm{Av}_n(213)}\#\{v\in V(G_\pi):\deg(v)=r\}.
\]
As a consequence, the proportion of degree-$4$ vertices tends to $1$ as $n\to\infty$, with a leading correction of order $n^{-1/2}$;
this contrasts with the unrestricted case \citep{BK}, where the limiting proportion for degree $4$ equals $1/2$.
Moreover, by reversal, all results transfer automatically to $\mathrm{Av}_n(312)$.

Our strategy is entirely Catalan: we start from the decomposition induced by the position of the minimum $1$,
translate local contributions into gluing identities and functional equations for ordinary generating functions algebraic over
$\mathbb{Q}(x,\sqrt{1-4x})$, and close the system using global identities (total number of vertices and sum of degrees).
The paper is organized as follows: in Section~\ref{sec:prelim} we fix definitions and local degree criteria; in Section~\ref{sec:decomp}
we establish the decomposition by the position of $1$; in Section~\ref{sec:H} we close $H_n$; in Section~\ref{sec:H5} we introduce the global
degree totals and the identities that close the system; in Section~\ref{sec:H6} we determine $Q_1(n)$; in Section~\ref{sec:H7} we determine
$Q_4(n)$; in Section~\ref{sec:H8} we package the closed forms for $Q_1(n),\dots,Q_4(n)$; in Section~\ref{sec:H9} we extract limiting
proportions and leading corrections and compare with \citep{BK}; and finally we conclude with a brief discussion of remaining cases and
further directions.

\section{Preliminaries}\label{sec:prelim}

We collect the basic definitions and notation: the grid graph $G_\pi$, local degree-by-level criteria,
Catalan enumeration for $\mathrm{Av}_n(213)$ and its generating function, and the reversal symmetry used to transfer results to $\mathrm{Av}_n(312)$.

\subsection{Grid graphs}

\begin{definition}
Let $\pi=\pi_1\cdots \pi_n\in S_n$. The \emph{grid graph} $G_\pi$ has vertex set
\[
V(G_\pi)=\{(i,j):1\le i\le n,\ 1\le j\le \pi_i\}.
\]
Two vertices $(i,j)$ and $(i',j')$ are adjacent if:
(i) $i=i'$ and $|j-j'|=1$ (a vertical edge), or
(ii) $|i-i'|=1$ and $j=j'$ (a horizontal edge), provided both vertices exist.
\end{definition}

For $n\ge2$, every vertex of $G_\pi$ has degree in $\{1,2,3,4\}$ \citep{BK}.
We write $\deg(v)$ for the degree of $v$, and $E_{\rm vert}(G_\pi)$, $E_{\rm hor}(G_\pi)$ for the sets of vertical and horizontal edges.

\subsection{Local degree criteria by level}

\begin{lemma}\label{lem:degByLevel}
Let $\pi\in S_n$ and consider a column $i$ of height $b=\pi_i$.

\smallskip
\noindent\textup{(i) If $2\le i\le n-1$ (internal column) and $a=\pi_{i-1}$, $c=\pi_{i+1}$, then for $1\le s\le b$}
\begin{equation}\label{eq:degByLevelInternal}
\deg(i,s)=\mathbf{1}_{\{s>1\}}+\mathbf{1}_{\{s<b\}}+\mathbf{1}_{\{s\le a\}}+\mathbf{1}_{\{s\le c\}}.
\end{equation}

\smallskip
\noindent\textup{(ii) If $i=1$ (left external column) and $b=\pi_1$, $c=\pi_2$, then for $1\le s\le b$}
\[
\deg(1,s)=\mathbf{1}_{\{s>1\}}+\mathbf{1}_{\{s<b\}}+\mathbf{1}_{\{s\le c\}}.
\]
\textup{Similarly, if $i=n$ and $a=\pi_{n-1}$, then}
\[
\deg(n,s)=\mathbf{1}_{\{s>1\}}+\mathbf{1}_{\{s<b\}}+\mathbf{1}_{\{s\le a\}}.
\]
\end{lemma}

\begin{proof}
In any column, vertical neighbors exist exactly when $s>1$ (a lower neighbor) and $s<b$ (an upper neighbor).
In an internal column there are two possible horizontal neighbors: to the left if $s\le a$ and to the right if $s\le c$.
In an external column there is only one possible horizontal neighbor. Summing these contributions yields the stated formulas.
\end{proof}

\begin{corollary}\label{cor:deg4FromLevel}
In an internal column with triple $(a,b,c)$, the number of degree-$4$ vertices in the middle column is
\[
\max\{0,\min(a,c,b-1)-1\}.
\]
\end{corollary}

\begin{proof}
By \eqref{eq:degByLevelInternal}, a level $s$ has degree $4$ if and only if $2\le s\le b-1$ and also $s\le a$ and $s\le c$.
This is equivalent to $2\le s\le \min(a,c,b-1)$, and the count is immediate.
\end{proof}

\begin{corollary}\label{cor:topInternal}
Let $2\le i\le n-1$. The top vertex $(i,\pi_i)$ has degree $1$ if and only if $\pi_i\ge 2$ and
$\pi_i>\max\{\pi_{i-1},\pi_{i+1}\}$.
\end{corollary}

\begin{proof}
Apply \eqref{eq:degByLevelInternal} with $s=b$. Then $\mathbf{1}_{\{s<b\}}=0$ and $\mathbf{1}_{\{s>1\}}=1$ if and only if $b\ge2$.
Moreover $\mathbf{1}_{\{s\le a\}}=\mathbf{1}_{\{b\le a\}}$ and $\mathbf{1}_{\{s\le c\}}=\mathbf{1}_{\{b\le c\}}$.
Hence the degree equals $1$ exactly when $b\ge2$ and $b>a$ and $b>c$.
\end{proof}

\subsection{Pattern avoidance, Catalan numbers, and reversal}

Let $\mathrm{Av}_n(213)\subseteq S_n$ be the set of permutations avoiding the pattern $213$.
It is well known (see, e.g., \citep[Ch.~4]{Bona}) that $|\mathrm{Av}_n(213)|$ is the $n$th Catalan number:
\[
C_n:=|\mathrm{Av}_n(213)|=\mathrm{Cat}_n=\frac{1}{n+1}\binom{2n}{n}.
\]
Write $B_n:=\binom{2n}{n}$, so $C_n=B_n/(n+1)$. The Catalan ordinary generating function is
\begin{equation}\label{eq:Cx}
C(x)=\sum_{n\ge 0} C_n x^n=\frac{1-\sqrt{1-4x}}{2x},
\end{equation}
and it satisfies $C(x)=1+xC(x)^2$.

For $\pi=\pi_1\cdots\pi_n\in S_n$, define its \emph{reversal} by
\[
\mathrm{rev}(\pi):=\pi_n\cdots\pi_1.
\]

\begin{lemma}\label{lem:reversalIso}
For every $\pi\in S_n$, the graphs $G_\pi$ and $G_{\mathrm{rev}(\pi)}$ are isomorphic via
\[
\varphi:V(G_\pi)\to V(G_{\mathrm{rev}(\pi)}),\qquad \varphi(i,s)=(n+1-i,s).
\]
In particular, for every $r\in\{1,2,3,4\}$ we have $Q_r(\pi)=Q_r(\mathrm{rev}(\pi))$ and also $H(\pi)=H(\mathrm{rev}(\pi))$.
\end{lemma}

\begin{proof}
The map $\varphi$ is a bijection and preserves adjacencies: vertical edges remain within the same column at consecutive levels,
and horizontal edges are reflected between adjacent columns. Hence degrees and degree counts are preserved. Moreover,
\[
H(\mathrm{rev}(\pi))=\sum_{t=1}^{n-1}\min\{\pi_{n+1-t},\pi_{n-t}\}=H(\pi).
\]
\end{proof}

\begin{corollary}\label{cor:312}
The map $\pi\mapsto \mathrm{rev}(\pi)$ restricts to a bijection $\mathrm{Av}_n(213)\leftrightarrow \mathrm{Av}_n(312)$.
Consequently, all results of this paper for $\mathrm{Av}_n(213)$ hold verbatim for $\mathrm{Av}_n(312)$, with the same closed formulas.
\end{corollary}

\begin{proof}
Reversal sends the pattern $213$ to $312$ (and conversely). The rest follows from Lemma~\ref{lem:reversalIso}.
\end{proof}

\section{Catalan decomposition by the position of the minimum}\label{sec:decomp}

We establish the Catalan block decomposition of $213$-avoiding permutations via the position of the minimum entry $1$,
which is the structural input for all subsequent gluing identities and generating-function equations.

The key structural fact is that, in a $213$-avoiding permutation, all entries to the left of $1$ are larger than all entries to the right.

\begin{lemma}\label{lem:decomp213}
Let $\pi\in\mathrm{Av}_n(213)$ and let $k$ be the position of $1$, i.e.\ $\pi_k=1$.
Then every entry to the left of $1$ is larger than every entry to the right of $1$.
In particular, if $i=k-1$ and $j=n-k$, there exist $\alpha\in\mathrm{Av}_i(213)$ and $\beta\in\mathrm{Av}_j(213)$ such that
\begin{equation}\label{eq:decomp}
\pi=(\alpha+j+1)\ 1\ (\beta+1),
\end{equation}
where $\alpha+j+1$ means adding $j+1$ to each entry of $\alpha$, and $\beta+1$ means adding $1$ to each entry of $\beta$.
Conversely, every concatenation of the form \eqref{eq:decomp} lies in $\mathrm{Av}_n(213)$.
\end{lemma}

\begin{proof}
If there were positions $p<k<q$ with $\pi_p<\pi_q$, then the triple $(\pi_p,1,\pi_q)$ would realize the pattern $213$:
indeed, $1$ is the smallest, $\pi_q$ the largest, and $\pi_p$ lies strictly between them, a contradiction.
Hence every entry left of $1$ is larger than every entry right of $1$.
Standardizing the left and right blocks yields \eqref{eq:decomp}.
The converse is immediate: in \eqref{eq:decomp} no $213$ pattern can cross the $1$, and within each block avoidance is preserved under standardization.
\end{proof}

In particular, we recover the usual Catalan convolution:
\[
C_n=\sum_{i+j=n-1} C_i C_j \qquad (n\ge 1).
\]

\section{Horizontal edges}\label{sec:H}

We derive a gluing identity for the horizontal-edge statistic $H(\pi)$ under the Catalan decomposition, translate it into a functional equation
for $H(x)$, and solve it to obtain a closed form for $H_n$.

For $\pi\in S_n$ define
\[
H(\pi)=\sum_{t=1}^{n-1}\min\{\pi_t,\pi_{t+1}\},\qquad
H_n:=\sum_{\pi\in\mathrm{Av}_n(213)} H(\pi).
\]
For a statement $A$ we write $\mathbf{1}_{\{A\}}\in\{0,1\}$ for its indicator.

\subsection{A gluing identity and a functional equation}

\begin{lemma}\label{lem:glueH213}
Let $\pi\in\mathrm{Av}_n(213)$ with decomposition $\pi=(\alpha+j+1)\,1\,(\beta+1)$ as in \eqref{eq:decomp},
where $|\alpha|=i$, $|\beta|=j$, and $i+j=n-1$. Then
\begin{equation}\label{eq:glueH213}
H(\pi)=H(\alpha)+H(\beta)
+\mathbf{1}_{\{i\ge 2\}}(i-1)(j+1)+\mathbf{1}_{\{j\ge 2\}}(j-1)+\mathbf{1}_{\{i\ge 1\}}+\mathbf{1}_{\{j\ge 1\}}.
\end{equation}
\end{lemma}

\begin{proof}
Split the sum defining $H(\pi)$ into adjacent pairs according to their position relative to the decomposition
$\pi=(\alpha+j+1)\,1\,(\beta+1)$.

If $i\ge2$, the pairs $(t,t+1)$ with $1\le t\le i-1$ lie entirely in the left block. For each such pair,
\[
\min\{\alpha_t+j+1,\alpha_{t+1}+j+1\}=(j+1)+\min\{\alpha_t,\alpha_{t+1}\},
\]
so the total contribution is $H(\alpha)+(i-1)(j+1)$. If $i\le1$ there are no such pairs, captured by $\mathbf{1}_{\{i\ge2\}}(i-1)(j+1)$.

Similarly, if $j\ge2$, the internal pairs in the right block contribute
\[
H(\beta)+(j-1),
\]
since $\min\{\beta_t+1,\beta_{t+1}+1\}=1+\min\{\beta_t,\beta_{t+1}\}$. If $j\le1$ there is no contribution, captured by $\mathbf{1}_{\{j\ge2\}}(j-1)$.

Finally, consider the pairs involving the column of $1$. If $i\ge1$, the pair between the last column of the left block and $1$
contributes $\min\{\alpha_i+j+1,1\}=1$; and if $j\ge1$, the pair between $1$ and the first column of the right block contributes
$\min\{1,\beta_1+1\}=1$. This yields $\mathbf{1}_{\{i\ge1\}}+\mathbf{1}_{\{j\ge1\}}$.

Collecting terms gives \eqref{eq:glueH213}.
\end{proof}

\begin{proposition}\label{prop:HFE213}
Let $H(x):=\sum_{n\ge 1} H_n x^n$. Then
\begin{equation}\label{eq:HFE213}
(1-2xC(x))\,H(x)=x\bigl(xC'(x)-C(x)+1\bigr)\bigl(xC'(x)+2C(x)\bigr)+2\bigl(C(x)-1-xC(x)\bigr).
\end{equation}
\end{proposition}

\begin{proof}
Sum \eqref{eq:glueH213} over all $\pi\in\mathrm{Av}_n(213)$, grouping by $i+j=n-1$ and the choices
$\alpha\in\mathrm{Av}_i(213)$, $\beta\in\mathrm{Av}_j(213)$.

The contributions $H(\alpha)$ and $H(\beta)$ yield the convolution
$\sum_{i+j=n-1}(C_jH_i+C_iH_j)$, which translates into $2xC(x)H(x)$ in ordinary generating functions.

Terms depending only on lengths become products of OGFs:
\[
\sum_{i\ge0}\sum_{j\ge0}\mathbf{1}_{\{i\ge2\}}(i-1)C_i x^i\cdot \sum_{j\ge0}(j+1)C_j x^j
= \bigl(xC'(x)-C(x)+1\bigr)\bigl(xC'(x)+C(x)\bigr),
\]
and similarly,
\[
\sum_{i\ge0}\sum_{j\ge0} C_i x^i\cdot \mathbf{1}_{\{j\ge2\}}(j-1)C_j x^j
= C(x)\bigl(xC'(x)-C(x)+1\bigr).
\]
Finally,
\[
\sum_{n\ge1}\Bigl(\sum_{i+j=n-1}(\mathbf{1}_{\{i\ge1\}}+\mathbf{1}_{\{j\ge1\}})C_iC_j\Bigr)x^n
=2\sum_{n\ge1}(C_n-C_{n-1})x^n=2\bigl(C(x)-1-xC(x)\bigr).
\]
Combining all contributions and moving $2xC(x)H(x)$ to the left yields \eqref{eq:HFE213}.
\end{proof}

\subsection{Closed form}

\begin{theorem}\label{thm:H}
For $n\ge1$,
\[
H_n=\frac{n}{2}\binom{2n}{n}-4^{\,n-1}.
\]
Equivalently, the ordinary generating function $H(x)=\sum_{n\ge1}H_nx^n$ is
\begin{equation}\label{eq:HxClosed213}
H(x)=x(1-4x)^{-3/2}-x(1-4x)^{-1}.
\end{equation}
\end{theorem}

\begin{proof}
Starting from \eqref{eq:HFE213}, substitute the Catalan identities
$C(x)=\frac{1-\sqrt{1-4x}}{2x}$ and $1-2xC(x)=\sqrt{1-4x}$,
and simplify to obtain \eqref{eq:HxClosed213}.

To extract coefficients, use
\[
(1-4x)^{-1}=\sum_{m\ge0}4^m x^m,\qquad
(1-4x)^{-1/2}=\sum_{m\ge0}\binom{2m}{m}x^m,
\]
and
\[
(1-4x)^{-3/2}=\frac{1}{2}\frac{d}{dx}(1-4x)^{-1/2}
=\sum_{m\ge0}(2m+1)\binom{2m}{m}x^m.
\]
Hence,
\[
[x^n]\,x(1-4x)^{-3/2}=(2n-1)\binom{2n-2}{n-1}
=\frac{n}{2}\binom{2n}{n},
\qquad
[x^n]\,x(1-4x)^{-1}=4^{n-1},
\]
which gives the claimed formula for $H_n$.
\end{proof}

\section{Degree counts: global definitions}\label{sec:H5}

We introduce the global degree totals $Q_r(n)$ and record two universal global identities: the total number of vertices and the sum of degrees,
which later close the system for $(Q_1,Q_2,Q_3,Q_4)$.

For $r\in\{1,2,3,4\}$ define
\[
Q_r(n):=\sum_{\pi\in\mathrm{Av}_n(213)} \#\{v\in V(G_\pi):\deg(v)=r\}.
\]

We also use two global identities (valid for any subclass of $S_n$).

\begin{proposition}\label{prop:VnSig213}
For $n\ge1$,
\begin{align}
V_n &:=\sum_{\pi\in\mathrm{Av}_n(213)} |V(G_\pi)|
= |\mathrm{Av}_n(213)|\cdot \frac{n(n+1)}{2} = \frac{n}{2}\binom{2n}{n},\label{eq:Vn}\\
\Sigma_n &:=\sum_{\pi\in\mathrm{Av}_n(213)} \sum_{v\in V(G_\pi)}\deg(v)
= |\mathrm{Av}_n(213)|\cdot n(n-1)+2H_n
= \frac{2n^2}{n+1}\binom{2n}{n}-2\cdot 4^{\,n-1}.\label{eq:SigmaClosed}
\end{align}
In particular, for $n\ge2$,
\begin{equation}\label{eq:LinSystem}
Q_1(n)+Q_2(n)+Q_3(n)+Q_4(n)=V_n,\qquad
Q_1(n)+2Q_2(n)+3Q_3(n)+4Q_4(n)=\Sigma_n.
\end{equation}
\end{proposition}

\begin{proof}
For each $\pi\in S_n$, the total number of vertices is $|V(G_\pi)|=\sum_{i=1}^n \pi_i=\binom{n+1}{2}$.
Summing over $\mathrm{Av}_n(213)$ gives \eqref{eq:Vn}.

For $\Sigma_n$, use the handshaking lemma: $\sum_v\deg(v)=2|E(G_\pi)|$.
Moreover, $|E_{\rm vert}(G_\pi)|=\sum_{i=1}^n(\pi_i-1)=\binom{n}{2}$ is constant on $S_n$,
whereas $|E_{\rm hor}(G_\pi)|=H(\pi)$. Summing over $\mathrm{Av}_n(213)$ yields
\[
\Sigma_n=2\Bigl(|\mathrm{Av}_n(213)|\cdot \binom{n}{2}+H_n\Bigr),
\]
and substituting $H_n$ from Theorem~\ref{thm:H} gives \eqref{eq:SigmaClosed}.
The system \eqref{eq:LinSystem} is the direct translation of \eqref{eq:Vn} and \eqref{eq:SigmaClosed} in terms of $Q_r(n)$.
\end{proof}

\section{Degree-$1$ vertices}\label{sec:H6}

We compute $Q_1(n)$ by combining a local characterization of degree-$1$ vertices on the boundary columns with Catalan boundary statistics
(initial descents and final ascents) and a generating-function closure for the internal contribution.

\subsection{A local criterion in external columns}

The following criterion matches \citet{BK} and also follows directly from Lemma~\ref{lem:degByLevel}.

\begin{lemma}\label{lem:deg1external}
Let $\pi\in S_n$ and let $G_\pi$ be its grid graph.
In the first column, the number of degree-$1$ vertices equals
\[
\mathbf{1}_{\{\pi_1=1\}}+\mathbf{1}_{\{\pi_1>\pi_2\}},
\]
and in the last column it equals
\[
\mathbf{1}_{\{\pi_n=1\}}+\mathbf{1}_{\{\pi_{n-1}<\pi_n\}}.
\]
\end{lemma}

\begin{proof}
Apply Lemma~\ref{lem:degByLevel}.
In the first column, with $b=\pi_1$ and $c=\pi_2$, for $s<b$ we have $\mathbf{1}_{\{s<b\}}=1$ and also $\mathbf{1}_{\{s\le c\}}=1$,
so $\deg(1,s)\ge2$. Hence only the top vertex $(1,b)$ can have degree $1$, and
\[
\deg(1,b)=\mathbf{1}_{\{b>1\}}+\mathbf{1}_{\{b\le c\}}.
\]
Thus $\deg(1,b)=1$ iff $b=1$ or $b>c$, i.e.\ $\pi_1=1$ or $\pi_1>\pi_2$.

The last column is identical: with $a=\pi_{n-1}$ and $b=\pi_n$, only $(n,b)$ can have degree $1$, and
\[
\deg(n,b)=\mathbf{1}_{\{b>1\}}+\mathbf{1}_{\{b\le a\}}.
\]
Hence $\deg(n,b)=1$ iff $b=1$ or $a<b$, i.e.\ $\pi_n=1$ or $\pi_{n-1}<\pi_n$.
\end{proof}

\subsection{Boundary counts in $\mathrm{Av}_n(213)$}

\begin{lemma}\label{lem:frontback213}
For $m\ge2$, define
\[
D_m:=\#\{\gamma\in\mathrm{Av}_m(213):\gamma_1>\gamma_2\},\qquad
A_m:=\#\{\gamma\in\mathrm{Av}_m(213):\gamma_{m-1}<\gamma_m\}.
\]
Then $D_m=A_m=C_{m-1}$.
\end{lemma}

\begin{proof}
We prove $D_m=C_{m-1}$; the proof for $A_m$ is analogous by symmetry (reversing left--right).

Let $\gamma\in\mathrm{Av}_m(213)$ and write $\gamma=(\alpha+j+1)\,1\,(\beta+1)$ with $|\alpha|=i$, $|\beta|=j$, and $i+j=m-1$
(Lemma~\ref{lem:decomp213}). If $i=0$, then $\gamma_1=1$ and there is no initial descent. If $i=1$, then necessarily
$\gamma_1=m$ and $\gamma_2=1$, so an initial descent always occurs and the choice reduces to $\beta\in\mathrm{Av}_{m-2}(213)$,
contributing $C_{m-2}$ possibilities. If $i\ge2$, then the first two entries of $\gamma$ lie in the left block and
$\gamma_1>\gamma_2$ is equivalent to $\alpha_1>\alpha_2$. Therefore, for $m\ge2$,
\[
D_m=C_{m-2}+\sum_{i=2}^{m-1} D_i\,C_{m-1-i}.
\]
Using induction and the Catalan convolution $\sum_{p=0}^{m-2}C_pC_{m-2-p}=C_{m-1}$ yields $D_m=C_{m-1}$.
\end{proof}

\subsection{External and internal columns}

\begin{lemma}\label{lem:Q1ext}
For $n\ge2$, the total number of degree-$1$ vertices contributed by the first and last columns satisfies
\[
Q_1^{\rm ext}(n)=4\,C_{n-1}.
\]
\end{lemma}

\begin{proof}
By Lemma~\ref{lem:deg1external},
\[
Q_1^{\rm ext}(n)=\sum_{\pi\in\mathrm{Av}_n(213)}\Bigl(
\mathbf{1}_{\{\pi_1=1\}}+\mathbf{1}_{\{\pi_1>\pi_2\}}+\mathbf{1}_{\{\pi_n=1\}}+\mathbf{1}_{\{\pi_{n-1}<\pi_n\}}
\Bigr).
\]
Each of the four sums equals $C_{n-1}$. Indeed, $\pi_1=1$ forces $1$ to be in the first position; by Lemma~\ref{lem:decomp213}
this means $i=0$ and leaves $\beta\in\mathrm{Av}_{n-1}(213)$ free, so $\sum_\pi \mathbf{1}_{\{\pi_1=1\}}=C_{n-1}$. Similarly, $\pi_n=1$ forces
$j=0$ and leaves $\alpha\in\mathrm{Av}_{n-1}(213)$ free, giving $\sum_\pi \mathbf{1}_{\{\pi_n=1\}}=C_{n-1}$.

For the comparative events, Lemma~\ref{lem:frontback213} gives
$\sum_\pi \mathbf{1}_{\{\pi_1>\pi_2\}}=D_n=C_{n-1}$ and $\sum_\pi \mathbf{1}_{\{\pi_{n-1}<\pi_n\}}=A_n=C_{n-1}$.
Hence $Q_1^{\rm ext}(n)=4C_{n-1}$.
\end{proof}

\begin{lemma}\label{lem:Q1int}
For $n\ge3$, the total number of degree-$1$ vertices in internal columns satisfies
\[
Q_1^{\rm int}(n)=(n-2)\,C_{n-1}.
\]
\end{lemma}

\begin{proof}
Let $P_n:=Q_1^{\rm int}(n)$. Fix $\pi\in\mathrm{Av}_n(213)$ with decomposition
$\pi=(\alpha+j+1)\,1\,(\beta+1)$, where $|\alpha|=i$, $|\beta|=j$, and $i+j=n-1$.

Internal peaks of $\pi$ split into those entirely within $\alpha$, those entirely within $\beta$, and possible peaks created at the boundary
with the column of $1$. The first two contributions yield, after summing over all choices, the convolution term
$\sum_{i+j=n-1}(C_jP_i+C_iP_j)=2\sum_{i+j=n-1}C_jP_i$.

For the boundary: if $i\ge2$, the last column of the $\alpha$-block is an internal peak of $\pi$ iff it exceeds its left neighbor,
since its right neighbor is $1$; this is equivalent to $\alpha$ ending with a final ascent. Similarly, if $j\ge2$, the first column of the
$\beta$-block is an internal peak of $\pi$ iff $\beta$ begins with an initial descent. By Lemma~\ref{lem:frontback213},
the number of $\alpha\in\mathrm{Av}_i(213)$ with a final ascent is $A_i=C_{i-1}$, and the number of $\beta\in\mathrm{Av}_j(213)$ with an initial descent is
$D_j=C_{j-1}$. Summing over all choices yields the additional term
\[
\sum_{i+j=n-1}(C_jA_i+C_iD_j)=2\sum_{i+j=n-1}C_jC_{i-1}.
\]
By Catalan convolution,
\[
\sum_{i+j=n-1}C_jC_{i-1}=\sum_{p=0}^{n-2}C_pC_{n-2-p}=C_{n-1},
\]
and hence, for $n\ge3$,
\[
P_n=2\sum_{i+j=n-1}C_jP_i+2C_{n-1}.
\]

Passing to OGFs, with $P(x):=\sum_{n\ge0}P_nx^n$ and $P_0=P_1=P_2=0$, the recurrence becomes
\[
P(x)=2xC(x)P(x)+2x\bigl(C(x)-1-x\bigr),
\]
so
\[
P(x)=\frac{2x\bigl(C(x)-1-x\bigr)}{1-2xC(x)}=x\bigl(xC'(x)-C(x)+1\bigr).
\]
Extracting coefficients gives $P_n=(n-2)C_{n-1}$ for $n\ge3$.
\end{proof}

\begin{corollary}\label{cor:Q1}
For $n\ge2$,
\[
Q_1(n)=Q_1^{\rm ext}(n)+Q_1^{\rm int}(n)=(n+2)\,C_{n-1}
=\frac{n+2}{2(2n-1)}\binom{2n}{n}.
\]
\end{corollary}

\begin{proof}
Sum Lemmas~\ref{lem:Q1ext} and \ref{lem:Q1int}, and use
$C_{n-1}=\frac{1}{n}\binom{2n-2}{n-1}=\frac{1}{2(2n-1)}\binom{2n}{n}$.
\end{proof}

\section{Degree-$4$ vertices}\label{sec:H7}

We obtain $Q_4(n)$ using a local degree-$4$ criterion together with a uniform-shift analysis under the Catalan gluing,
leading to a closed generating function and an explicit coefficient formula.

\subsection{Uniform shifts}

The local criterion for counting degree-$4$ vertices in an internal column was established in Corollary~\ref{cor:deg4FromLevel}.
We now record the effect of uniform shifts in a triple of columns, which will be the gluing ingredient for the global closure.

\begin{lemma}\label{lem:shift}
Let an internal column have height $b$ with neighbors of heights $a$ and $c$.
If one replaces $(a,b,c)$ by $(a+t,b+t,c+t)$ with $t\ge1$, then the number of degree-$4$ vertices
in the middle column increases by
\[
\begin{cases}
t, & b\ge 2,\\
t-1, & b=1.
\end{cases}
\]
\end{lemma}

\begin{proof}
By Corollary~\ref{cor:deg4FromLevel}, before the shift the number of degree-$4$ vertices in the middle column is
$\max\{0,\min(a,c,b-1)-1\}$, and after the shift it is
\[
\max\{0,\min(a+t,c+t,(b+t)-1)-1\}=\max\{0,\,t+\min(a,c,b-1)-1\}.
\]
If $b\ge2$, then $b-1\ge1$ and since $a,c\ge1$ we have $\min(a,c,b-1)\ge1$, so the increment is exactly $t$.
If $b=1$, the first term is $0$ and the second is $\max\{0,t-1\}=t-1$.
\end{proof}

\subsection{Closure for $Q_4(n)$}

Let $Q_4(n)$ be the global total of degree-$4$ vertices as $\pi$ ranges over $\mathrm{Av}_n(213)$.
There are no degree-$4$ vertices in external columns, so $Q_4(n)$ equals the internal contribution.

Write $(t)_+:=\max\{t,0\}$. Also define, for $m\ge0$,
\[
J_m:=\#\{\gamma\in\mathrm{Av}_m(213):\text{the entry }1\text{ lies in an internal position}\}
=\begin{cases}
C_m-2C_{m-1}, & m\ge 2,\\
0, & m\le 1.
\end{cases}
\]

\begin{theorem}\label{thm:Q4}
For $n\ge2$,
\[
Q_4(n)=\frac{4n^3-7n^2+29n-20}{4(n+1)(2n-1)}\binom{2n}{n}-7\cdot 4^{\,n-2}.
\]
\end{theorem}

\begin{proof}
Fix $\pi\in\mathrm{Av}_n(213)$ and write $\pi=(\alpha+j+1)\,1\,(\beta+1)$, with $|\alpha|=i$, $|\beta|=j$, and $i+j=n-1$
(Lemma~\ref{lem:decomp213}). By Corollary~\ref{cor:deg4FromLevel}, an internal column cannot contain degree-$4$ vertices if one of its
neighbors has height $1$. In particular, the columns adjacent to $1$ in $\pi$ contribute nothing to $Q_4(\pi)$.
Thus every contribution comes from internal columns fully contained in $\alpha$ or fully contained in $\beta$.

When embedding the left block into $\pi$, all its heights increase uniformly by $t=j+1$.
Applying Lemma~\ref{lem:shift} to each internal column of $\alpha$, the shift creates $t$ new degree-$4$ vertices per internal column,
except possibly at the internal column whose original height is $1$, where it creates $t-1$.
Since $\alpha$ has $(i-2)_+$ internal columns and there is exactly one entry of value $1$ in $\alpha$, the total increment is
$(j+1)(i-2)_+ - \mathbf{1}_{\{1\text{ is internal in }\alpha\}}$.
For $\beta$, the shift is $t=1$, and the total increment is
$(j-2)_+ - \mathbf{1}_{\{1\text{ is internal in }\beta\}}$.
Therefore, for each pair $(\alpha,\beta)$,
\begin{equation}\label{eq:Q4glue}
Q_4(\pi)=Q_4(\alpha)+Q_4(\beta)+(j+1)(i-2)_+ +(j-2)_+ - \mathbf{1}_{\{1\text{ internal in }\alpha\}}-\mathbf{1}_{\{1\text{ internal in }\beta\}}.
\end{equation}

Summing \eqref{eq:Q4glue} over $\alpha\in\mathrm{Av}_i(213)$, $\beta\in\mathrm{Av}_j(213)$ and then over $i+j=n-1$ yields a Catalan convolution recurrence.
Passing to OGFs leads to the linear functional equation
\begin{equation}\label{eq:Q4FE213}
(1-2xC(x))Q_4(x)=xA(x)\bigl(C(x)+(xC(x))'\bigr)-2xC(x)J(x),
\end{equation}
where $Q_4(x)=\sum_{n\ge0}Q_4(n)x^n$, $A(x)=\sum_{m\ge0}(m-2)_+C_mx^m$, and $J(x)=\sum_{m\ge0}J_mx^m$.

From the definitions,
\[
A(x)=xC'(x)-2C(x)+2+x,\qquad J(x)=(1-2x)C(x)+x-1.
\]
Substitute these into \eqref{eq:Q4FE213} and use $C(x)=\frac{1-\sqrt{1-4x}}{2x}$ and $1-2xC(x)=\sqrt{1-4x}$.
After algebraic simplification, one obtains
\begin{equation}\label{eq:Q4xClosed}
Q_4(x)=\frac{8x^4-150x^3+157x^2-50x+5 + (36x^3-87x^2+40x-5)\sqrt{1-4x}}{2x(1-4x)^2}.
\end{equation}

Finally, \eqref{eq:Q4xClosed} can be written as a linear combination of $(1-4x)^{-1}$, $(1-4x)^{-2}$, and $(1-4x)^{-3/2}$,
and coefficients are extracted using standard expansions (see, e.g., \citep[Ch.~VI]{FS}).
This yields, for $n\ge2$,
\[
Q_4(n)=\frac{4n^3-7n^2+29n-20}{4(n+1)(2n-1)}\binom{2n}{n}-7\cdot4^{\,n-2}.
\]
\end{proof}

\section{Closed forms for $Q_1,\dots,Q_4$}\label{sec:H8}

We package the closed forms for $Q_1(n)$ and $Q_4(n)$ and then solve for $Q_2(n)$ and $Q_3(n)$ from the global linear identities,
yielding explicit formulas for all degree totals.

\begin{theorem}\label{thm:Qall213}
For $n\ge2$, with $B_n=\binom{2n}{n}$, we have
\begin{align*}
Q_1(n) &= \frac{n+2}{2(2n-1)}\,B_n,\\[2mm]
Q_2(n) &= \frac{(12n^2+44n-112)\,B_n+(2n^2+n-1)\,4^n}{16(n+1)(2n-1)},\\[2mm]
Q_3(n) &= \frac{(8n^2-96n+88)\,B_n+(6n^2+3n-3)\,4^n}{8(n+1)(2n-1)},\\[2mm]
Q_4(n) &= \frac{(4n^3-7n^2+29n-20)\,B_n}{4(n+1)(2n-1)}-7\cdot 4^{\,n-2}.
\end{align*}
\end{theorem}

\begin{proof}
The formula for $Q_1(n)$ is Corollary~\ref{cor:Q1}, and the formula for $Q_4(n)$ is Theorem~\ref{thm:Q4}.
For $Q_2(n)$ and $Q_3(n)$, use the linear system \eqref{eq:LinSystem} with $V_n$ and $\Sigma_n$ given by
\eqref{eq:Vn} and \eqref{eq:SigmaClosed}. Set
\[
S_1:=V_n-Q_1(n)-Q_4(n)=Q_2(n)+Q_3(n),\qquad
S_2:=\Sigma_n-Q_1(n)-4Q_4(n)=2Q_2(n)+3Q_3(n).
\]
Then $Q_3(n)=S_2-2S_1$ and $Q_2(n)=S_1-Q_3(n)$.
Substituting the closed expressions for $V_n,\Sigma_n,Q_1(n),Q_4(n)$ and simplifying gives the stated formulas.
\end{proof}

\begin{remark}
If $\pi$ is chosen uniformly at random from $\mathrm{Av}_n(213)$, then
\[
\mathbb{E}\bigl[\#\{v\in V(G_\pi):\deg(v)=r\}\bigr]=\frac{Q_r(n)}{C_n}\qquad (r=1,2,3,4),
\]
and also $\mathbb{E}[H(\pi)]=\dfrac{H_n}{C_n}$.
\end{remark}

\section{Asymptotics and comparison with the unrestricted case}\label{sec:H9}

We use standard asymptotics for central binomial coefficients to extract limiting proportions and leading corrections for $Q_r(n)/V_n$,
and we compare the resulting degree distribution in $\mathrm{Av}_n(213)$ with the unrestricted case studied in \citet{BK}.

Let $V_n=\frac{n}{2}B_n$ be the total number of vertices summed over $\mathrm{Av}_n(213)$.
Using the classical asymptotic $B_n\sim \dfrac{4^n}{\sqrt{\pi n}}$ (see, e.g., \citep[App.~A]{FS}) together with Theorem~\ref{thm:Qall213}, we obtain
\[
\frac{Q_1(n)}{V_n}\to 0,\qquad
\frac{Q_2(n)}{V_n}\to 0,\qquad
\frac{Q_3(n)}{V_n}\to 0,\qquad
\frac{Q_4(n)}{V_n}\to 1.
\]
Moreover, the subdominant term of $Q_4(n)$ is $-7\cdot 4^{n-2}$, and therefore
\[
\frac{Q_4(n)}{V_n}=1-\frac{7}{8}\sqrt{\pi}\,n^{-1/2}+O(n^{-1}),
\]
in contrast with the unrestricted case \citep{BK}, where the limiting proportion for degree $4$ equals $1/2$.

\begin{corollary}\label{cor:propsLeading}
Let $V_n=\frac{n}{2}B_n$. As $n\to\infty$,
\begin{align*}
\frac{Q_1(n)}{V_n} &= \frac{1}{2n}+O(n^{-2}),\\
\frac{Q_2(n)}{V_n} &= \frac{\sqrt{\pi}}{8}\,n^{-1/2}+O(n^{-1}),\\
\frac{Q_3(n)}{V_n} &= \frac{3\sqrt{\pi}}{4}\,n^{-1/2}+O(n^{-1}),\\
\frac{Q_4(n)}{V_n} &= 1-\frac{7\sqrt{\pi}}{8}\,n^{-1/2}+O(n^{-1}).
\end{align*}
\end{corollary}

\begin{proof}
Use $B_n\sim \dfrac{4^n}{\sqrt{\pi n}}$ (e.g.\ \citep[App.~A]{FS}) together with the closed forms in Theorem~\ref{thm:Qall213}.
For $Q_2$ and $Q_3$, the $4^n$ terms dominate the $B_n$ terms by a factor of $\sqrt{n}$; their coefficients tend to $1/16$ and $3/8$,
respectively, yielding the stated constants after dividing by $V_n\sim \dfrac{\sqrt{n}}{2\sqrt{\pi}}\,4^n$.
For $Q_4$, the main term matches $V_n$ and the first correction comes from $-7\cdot4^{n-2}=-(7/16)4^n$,
giving the constant $\frac{7\sqrt{\pi}}{8}$ at order $n^{-1/2}$.
Finally, $Q_1(n)\sim \frac14 B_n$ and dividing by $V_n\sim \frac n2 B_n$ gives $Q_1/V_n\sim (2n)^{-1}$.
\end{proof}

\section{Conclusions}\label{sec:concl}

We complete the program proposed by \citet{BK} for the Catalan class $\mathrm{Av}_n(213)$.
Starting from the decomposition induced by the position of the minimum $1$, we obtain closed functional equations for global statistics
of the grid graph $G_\pi$, and from them we deduce explicit formulas for both the total number of horizontal edges $H_n$
and the global degree totals $Q_r(n)$ ($r=1,2,3,4$).
The case $213$ exhibits behavior qualitatively different from the unrestricted one: avoidance rigidifies the Catalan gluing and concentrates
the degree distribution, so that the proportion of degree-$4$ vertices tends to $1$, with a leading deficit of order $n^{-1/2}$.
Moreover, by reversal, all results transfer immediately to $\mathrm{Av}_n(312)$.

On the other hand, the case $\mathrm{Av}_n(132)$ (and, by reversal, $\mathrm{Av}_n(231)$) has already been resolved recently in \citep{Hua132}.
Consequently, only the two remaining Catalan avoidance classes of length $3$,
\[
\mathrm{Av}_n(123)\quad \text{and}\quad \mathrm{Av}_n(321),
\]
remain open. Although they share Catalan enumeration, the structural constraints imposed by avoidance are substantially different and typically
require pattern-specific tools (a suitable choice of pivot, more rigid decompositions, and, in some instances, auxiliary boundary statistics)
to close the functional equations and obtain explicit formulas and asymptotic proportions.

\acknowledgements
\label{sec:ack}
This work was partially supported by the Vice-Rectorate for Research of the Universidad Nacional de San Crist\'obal de Huamanga (VRI-UNSCH).

\bibliographystyle{abbrvnat}
\bibliography{sample-dmtcs}
\label{sec:biblio}

\end{document}